\documentclass[12pt,reqno]{amsart}

\usepackage{amssymb,amsfonts,latexsym,amstext,epsfig,color}

\usepackage[left=2.5cm,top=2.5cm,right=2.5cm,bottom=2.5cm]{geometry}

\newtheorem{theorem}{Theorem}

\newtheorem{lemma}[theorem]{Lemma}
\newtheorem*{claim*}{Claim}

\newtheorem{corollary}[theorem]{Corollary}

\newcommand{\Nnn}{{\mathbb N}}

\newcommand{\R}{\mathbb R}
\newcommand{\N}{\mathbb N}
\newcommand{\Z}{\mathbb Z}
\newcommand{\PP}{\mathbb P}

\newcommand{\fracpart}[1]{\langle #1\rangle}
\newcommand{\gec}{g.e.c.}
\newcommand{\idf}{i.d.f.}
\newcommand{\iop}{i.o.p.}
\newcommand{\LARG}{\mathrm{LARG}}

\begin{document}

\title{Geometric random graphs and Rado sets in sequence spaces}
\author{Anthony Bonato}
\address{Department of Mathematics\\
Ryerson University\\
Toronto, ON\\
Canada, M5B 2K3} \email{abonato@ryerson.ca}
\author{Jeannette Janssen}
\address{Department of Mathematics and Statistics\\
Dalhousie University\\
Halifax, NS\\
Canada, B3H 3J5} \email{jeannette.janssen@dal.ca}
\author{Anthony Quas}
\address{Department of Mathematics and Statistics\\
University of Victoria \\
Victoria, BC\\
Canada, V8W 3R4} \email{aquas@uvic.ca}

\thanks{The authors gratefully acknowledge support from NSERC and Ryerson University}

\begin{abstract}
We consider a random geometric graph model, where pairs of vertices are points in a metric space and edges are formed independently with fixed probability $p$ between pairs within threshold
distance $\delta $. A countable dense set in a metric space is {\sl Rado} if this random model gives, with probability 1, a graph that is unique up to isomorphism. In earlier work, the first two
authors proved that, in finite dimensional spaces $\mathbb{R}^n$ equipped with the $\ell_{\infty}$ norm, all countable dense set satisfying a mild non-integrality condition are Rado. In this paper
we extend this result to infinite-dimensional spaces. If the underlying metric space is a separable Banach space, then we show in some cases that we can almost surely recover the Banach space from
such a geometric random graph. More precisely, we show that in the sequence spaces $c$ and $c_0$, for measures $\mu$ satisfying certain conditions, $\mu^\N$-almost all countable sets are
Rado. Moreover, with probability 1, in $c$ as in $c_0$, all graphs obtained from the random geometric model with a randomly chosen dense countable vertex set are isomorphic to each other. Finally, we show that representatives of the isomorphism classes obtained in this way  from $c$ and $c_0$  are non-isomorphic to each other, and also non-isomorphic to their counterparts obtained from finite dimensional spaces.

\end{abstract}
\maketitle

\section{Introduction}

The well known binomial random graph is a probability space
$G(\mathbb{N},p)$ that consists of graphs with vertices
$\mathbb{N}$, so that each distinct pair of integers is adjacent independently with
a fixed probability $ p\in (0,1)$. Over fifty years ago, Erd\H{o}s
and R\'{e}nyi \cite{er} discovered that with probability $1$, a graph
$G\in G(\mathbb{N},p)$ is isomorphic to a particular graph, the Rado graph $R$. Further, the almost sure
isomorphism class does not depend on $p$ for any fixed $0<p<1$.
A graph $G$ is \emph{existentially closed} (or \emph{e.c.}) if for
all finite disjoint sets of vertices $A$ and $B$ (one of which may
be empty), there is a vertex $z\notin A\cup B$ adjacent to all
vertices of $A$ and to no vertex of $B$. We say that $z$ is
\emph{correctly joined} to $A$ and $B$. The unique isomorphism type of a
countably infinite e.c.\ graph is named the \emph{infinite random graph},
or the \emph{Rado} graph, and is written $R$. See Chapter 6 of
\cite{bonato} and the surveys \cite{cam,cam1} for additional
background on $R$.

Isomorphism types distinct from $R$ but arising from infinite random geometric graphs
were first considered in \cite{bon2}. Consider a normed space $X$ with norm $\|\cdot\|$,
%distance function
%$$
%d:S \times S \rightarrow \mathbb{R},
%$$
a positive real number $\delta$, a countable subset $V$ of $X$,
and $p\in (0,1)$. The {\sl Local Area Random Graph
$\mathrm{LARG}(V,\delta,p)$} has vertices $V$, and for each
pair of vertices $u$ and
$v$ with $\|u-v\|<\delta$, an edge is added independently with
probability $p$. In other words, we consider a random,
constant-radius disk model on a subset of a metric space.
The LARG model
generalizes some well-known classes of random graphs. For example,
special cases of the $\mathrm{LARG}$ model include the 
random geometric graphs (where the space is Euclidean, the vertex set is chosen uniformly at random from a bounded subspace, and $p=1$), and the finite binomial random graph
$G(n,p)$ (where the base set is a metric space of finite diameter,
$D$, and $\delta> D$). Since $V$ is required to be countable and
we will focus on the case where $V$ is dense in $X$, we require $X$
to be separable.

A general question is the classification
of sets and metric spaces for which the LARG model, like the
random graph model, leads to a unique isomorphism type. The following
notation from \cite{bol1} is helpful. A countable dense set $V$
in a normed space $X$ is \emph{Rado} if for all $\delta >0$ and
$p\in (0,1)$, with probability $1$, $\mathrm{LARG}(V,\delta,p)$
generates a unique isomorphism type of graph. The set $V$ is
\emph{strongly non-Rado} if any two such graphs are,
with probability $1$, not isomorphic. A fundamental question
when studying graphs
generated by the LARG model is to determine which sets are Rado
or strongly non-Rado.

For a real number $1\le p \le \infty$ and $d\ge 1$ an integer,
the vector space $\mathbb{R}^d$ of dimension $d$ equipped with
the metric derived from the $p$-norm is denoted by $\ell_{p}^d$. If $p
=\infty$, then in \cite{bon2} it was shown that almost all countable
dense sets are Rado (here and for the rest of this section,
``almost all'' refers to a suitable measure constructed in the paper).
The proof used a back-and-forth argument, coupled with
a geometric analogue of the e.c.\ property (see the next section).
The unique countable limits in the $d$-dimensional case were named
$GR_d$; for a fixed $d$, these graphs are all isomorphic regardless
of the choice of $\delta$ or $p$. In contrast, it was also shown in
\cite{bon2} that if $p=2$, there are sets in $\ell_{p}^2$ which
are strongly non-Rado. The latter result was significantly
generalized in \cite{bol1}, which proved, using tools from
functional analysis, that if $X$ is a finite-dimensional normed
space not isometric
to $\ell_{\infty}^d$, then almost every random dense set $V$ is strongly non-Rado.

A question posed at the end of \cite{bol1} was to classify the Rado sets
in the infinite dimensional case. As discussed in \cite{bona}, the existence of Rado sets in sequence spaces is open and this problem is settled in the current paper.

Let $c$ be the space of all convergent, real sequences equipped with the $\ell _{\infty }$ norm. Let $c_0$ denote the subspace of $c$ consisting of sequences converging to 0. It is well known that
$c$ and $c_0$ are separable normed spaces. We prove in the following two theorems that $\mu^\N$-almost all countable dense sets in the space of convergent sequences $c$, and $\mu_0^\N$-almost all
countable subsets of the subspace $c_0$ (consisting of sequences in $c$ that converge to 0) are Rado, where $\mu$ and $\mu_0$ are measures on $c$ and $c_0$, respectively, satisfying certain natural
conditions.

The proofs of the following theorems are
deferred to Sections~\ref{csec} and \ref{c0sec}, respectively. See Section~\ref{randomsec} for the definitions of non-aligned, and fully supported measures, but we comment here that
any reasonable fully supported measure $\mu$ on $c$ and any reasonable fully supported product measure $\mu_0$ on $c_0$ has the non-aligned property. Briefly, a measure is non-aligned if each affine coordinate hyperplane is a null set, and by a product measure, we mean one
where the coordinates are mutually independent (but not identically distributed).

\begin{theorem}\label{mainone}
There exists a graph, $GR(c)$ such that, for every non-aligned fully supported measure $\mu$ on $c$,
for $\mu^\N$-almost every countable subset $V$ of $c$, for each $p\in(0,1)$,
$\LARG(V,1,p)$ is almost surely isomorphic to $GR(c)$.% where $\mu$ is as constructed in Subsection \ref{aa}.
\end{theorem}

\begin{theorem}\label{maintwo}
There exists a graph, $GR(c_0)$, such that for every non-aligned fully supported measure $\mu$ of product type on $c_0$,
for $\mu^\N$-almost every countable set $V$, for each $p\in(0,1)$, $\LARG(V,1,p)$ is almost surely
isomorphic to $GR(c_0)$.
\end{theorem}

The above theorems (as well as those of the previous paper \cite{bon2}) show that if $X$
in one of the Banach spaces $\ell_\infty^d$, $c$ or $c_0$, there exists a local Rado
graph $GR(X)$ such that for many countable dense subsets $V$ of $X$, and any
$p\in(0,1)$, $\LARG(V,1,p)$ is almost surely isomorphic to $GR(X)$. The following
theorem, proved in section \ref{sec:noniso}, shows that these local 
Rado graphs are mutually non-isomorphic: one can recover the underlying Banach space
from the local Rado graph.

\begin{theorem}\label{mainthree}
Suppose that $X$ and $Y$ are real Banach spaces with local Rado graphs 
$GR(X)$ and $GR(Y)$. If $GR(X)$ is isomorphic to $GR(Y)$, then
$X$ is isometrically equivalent to $Y$.
\end{theorem}

Throughout, all graphs considered are simple, undirected, and countable unless
otherwise stated. We will encounter two distinct notions of distance: metric distance
(derived from a given norm) and
graph distance. In a normed space, we write $\|u-v\|$ for the metric
distance of the points. For a graph $G$, we write $d_G(u,v)$ for the graph distance.
Given a normed space $X$, denote the (open)
\emph{ball of radius $\delta$ around $x$} (or $\delta$-\emph{ball}) by
\[
B_{\delta }(x)=\{u\in X:\|u-x\|<\delta \}.
\]
In large parts of the paper, we will take $\delta=1$.
In that case, we will use $B(x)$ instead of $B_1(x)$. A subset $V$ is \emph{dense}
in $X$ if for every point $x\in X$, every ball around $x$ contains at least one
point from $V$. We use the notation $\lfloor t \rfloor$ for the
integer part of $t\in\R$, and let $\fracpart{t} = t- \lfloor t \rfloor $.

For $a\in\R$, we write $c_a$ for the subset of $c$ consisting of sequences converging to $a$.
If $x$ is a point of $c$, then we write $(x_n)_{n\in\N}$ for the terms of the sequence defining $x$. We
use the notation $x_\infty$ to denote $\lim_{n\rightarrow \infty }x_{n}$.
%AQ Fix here: lim n\to\infty instead of x\to\infty
We use the notation $x\sim_G y$
to denote that $x$ and $y$ are
adjacent in $G$. If it is clear from the context which graph is meant, then we omit the
subscript and write $x\sim y$.

For a reference on graph theory the reader is directed to \cite{diestel,west},
while \cite{bryant} is a reference on normed spaces. 
%For further reading on random geometric graphs, see \cite{bol,ellis,goel,jj,walters}, and the book \cite{mpen}.

\section{Geometrically Existentially Closed graphs} \label{geo}
Let $G=(V,E)$ be a graph whose vertices are points in the normed space $X$.
The graph $G$ is \emph{geometrically existentially closed}
\emph{at level $\delta$} (or $\delta$-\emph{\gec}) if for all
$\delta'$ so that $0<\delta'<\delta$, for all $x\in V$, and for all disjoint finite sets
$A$ and $B$ so that $A\cup B\subseteq B_\delta(x)$, there exists a vertex
$z\not\in A\cup B\cup\{x\}$ so that
\begin{itemize}
\item[(i)] $z$ is correctly joined to $A$ and $B$,
\item[(ii)] \hspace{0.01in} $A\cup B \subseteq B_{\delta}(z)$,
\item[(iii)] \hspace{0.03in} $\|x-z\|<\delta'$.
\end{itemize}
This definition implies that $V$ is dense in itself, since we may
choose $A$ and $B$ to be empty. Further, if $G$ is $\delta$-\gec,
then $G$ is $\delta'$-\gec\ for any $\delta'<\delta$. If $V$ is
dense in a normed space, then by scaling, without loss of generality,
we may assume that $\delta = 1$, and we refer to $1$-\gec\ simply as \gec

The \gec~property bears clear similarities with the e.c.\ property defined in the
introduction. The important differences are that a correctly joined vertex must
exist only for sets $A$ and $B$
which are contained in an open ball with radius $\delta$ and centre $x$, and it must
be possible to choose the vertex $z$ correctly joined to $A$ and $B$ arbitrarily
close to $x$; see Figure~1.

\begin{figure}[h]
\begin{center}
\epsfig{figure=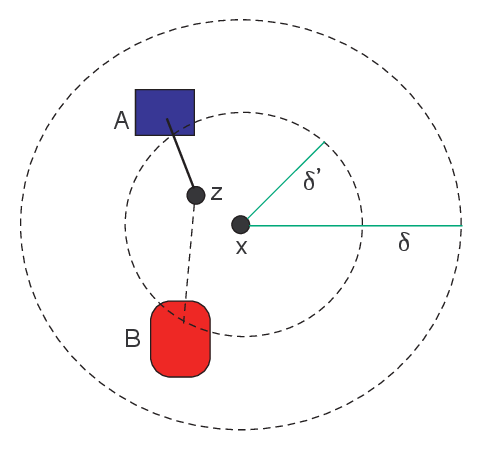,scale=0.75} \caption{The $\delta$-\gec\ property.}
\end{center}
\end{figure}

The following result demonstrates that graphs generated by the LARG model are
almost surely \gec\ Note that this result applies broadly to \emph{all} normed spaces,
including $c$ and $c_0$.

\begin{theorem}[\cite{bon2}]\label{random}
Let $(X,\|\cdot\|)$ be a normed space and $V$ a countable dense subset of $X$.
If $p\in (0,1)$, then with probability $1$, $\mathrm{LARG}(V,1,p)$ is \gec
\end{theorem}

Theorem~\ref{random} gives us a tool to prove that a set is Rado.
In particular, we use a back-and-forth argument to show that
\gec\ graphs on the set are isomorphic.

The \emph{length} of an edge $uv$ is the distance between its endpoints, i.e.\ $\| u-v\|$.
A countable graph that is \gec\ and is such that all its edges have length less than $\delta$, 
is called a \emph{geometric $\delta$-graph}.
By definition, a graph $G$ generated by $\mathrm{LARG}(V,1,p)$
has no edges of length more than $1$, and, if $V$ is countable and dense in $X$, then $G$
is almost surely \gec\ Thus, this random graph model generates geometric $1$-graphs.

The following important theorem
demonstrates that there exists a close
relationship between graph distance and metric distance in any graph that
is a geometric $1$-graph.

\begin{theorem}[\cite{bon2}]
\label{mot} Let $(X,\|\cdot\|)$ be a normed space and $G=(V,E)$ be
a geometric $1$-graph with $V\subseteq X$. Suppose that
$\overline{V}$ is convex. For any $u,v\in V$ so that $\|u-v\|\ge 1$, we have that
$$
d_G(u,v) = \lfloor \|u-v\|\rfloor +1.
$$
If $\|u-v\|<1$, then
$$
d_G(u,v) = \begin{cases}
0&\text{if $u=v$;}\\
1&\text{if $u\sim v$;}\\
2&\text{otherwise.}
\end{cases}
$$
\end{theorem}

Theorem~\ref{mot} directly leads to the following corollary.

\begin{corollary}[\cite{bon2}]\label{useful}
  If $\overline{V}$ and $\overline{W}$ are convex, and there are geometric
  1-graphs with vertex sets $V$ and $W$, respectively,
  which are isomorphic via $f$, then for
  every pair of vertices $u,v\in V$,
  $$
\lfloor \|u-v\|\rfloor=\lfloor \|f(u)-f(v)\|\rfloor.
$$
\end{corollary}

Corollary~\ref{useful} suggests the following generalization of isometry. Given metric
spaces $(S,d_S)$ and $(T,d_T)$, sets $V\subseteq S$ and $W\subseteq T$,  a
\emph{step-isometry} from $V$ to $W$
is a bijective map $f:V\rightarrow W$ with the property that for every
pair of vertices $u,v\in V$,
\[
\lfloor d_S(u,v)\rfloor=\lfloor d_T(f(u),f(v))\rfloor.
\]
Every isometry is a step-isometry, but the converse is false, in general.

\section{Random infinite sets in $c$ and $c_0$}\label{randomsec}

We will show that sets that are dense in $c$ or $c_a$ are
Rado if they satisfy certain mild properties, which we now define.
A subset $A\subseteq c$ is \emph{coordinate-wise integer distance free} or
\emph{\idf} if for all $n\in\N\cup\{\infty\}$, $x_n\not\in\Z$ for all
$x\in A$ and $x_n-y_n\not\in\Z$ for all distinct $x,y\in A$
(where $x_\infty$ denotes $\lim_{n\to\infty}x_n$ as mentioned above).
A subset $A\subseteq c_a$ is \emph{\idf}  if
for all $n\in\N$, $x_n\not\in\Z$ for all $x\in A$, and
$x_n-y_n\not\in\Z$ for all distinct $x,y\in A$. 
Note that if a set is \idf\ in $c$, it is \idf\ in $c_a$, but the converse is false, since all sequences in $c_a$ have the same limit.

When dealing with $c_a$, we need an additional second property (as will be explained in Section \ref{c0sec}.) A countable subset of $c_a$ $\{x^{(1)},x^{(2)},\ldots\}$ satisfying the \idf\ condition is said to have the \emph{independent order
property} (or \iop) if for any finite sub-collection of the points, $x^{(i_1)},\ldots,x^{(i_k)}$, and any finite collection of linear orders $\prec_1,\ldots,\prec_N$ of $\{1,\ldots,k\}$, there exist
distinct positions $j_1,\ldots,j_N\in\N$ such that for all $\ell \in \{ 1, 2, \cdots N\},$
\[
m\prec_\ell n \text{ if and only if }\fracpart{x^{(i_m)}_{j_\ell}}<
\fracpart{x^{(i_n)}_{j_\ell}}.
\]

Let $(\mu_n)$ be a sequence of probability measures on $\R$, each continuous and fully
supported. We construct a random element, $x$, of $\R^\N$ by independently
sampling each coordinate $x_n$ from the measure $\mu_n$. Providing we impose suitable
constraints on the $\mu_n$ (for example their distribution is symmetric around 0
and their variances are
summable), the resulting point, $x$, is almost surely in $c_0$. To obtain a random point
in $c$, we can just add the value of an independently sampled single continuous fully
supported random variable taking values in $\R$ to each coordinate of $x$.

\begin{theorem}
If $\mu$ is a measure on $X$ (one of $c$ or $c_0$) as described above, 
then $\mu^\N$-a.e. of $X^\N$ has the \idf\ and \iop\ properties.
\end{theorem}

The proof, given for specific constructions of suitable $\mu$, is contained in the appendix.

\section{Proof of Theorem~\ref{mainone}: almost all sets in $c$ are Rado}\label{csec}

We will show isomorphism of geometric 1-graphs in the spaces we consider
by inductively constructing an isomorphism. By Corollary \ref{useful}, we know that any isomorphism must be a step isometry. Thus, in the inductive construction, we need to make sure that
the constructed map is a step-isometry. The
following lemma gives sufficient conditions for a map on $c$ to be a step-isometry.

\begin{lemma}\label{step}
Let $A,B$ be \idf\ 
subsets of $c$, and let $f\colon A\to B$ be a bijection.
If the following conditions are satisfied, then $f$ is a step-isometry:
\begin{enumerate}
\item $\lfloor f(x)_i\rfloor = \lfloor x_i\rfloor$ for each $x\in A$ and each $i\in\N$;
\item For each $i\in\N$ and $x,y\in A$, we have that $\fracpart{f(x)_i} < \fracpart{f(y)_i}$ if and only if
$\fracpart{x_i}<\fracpart{y_i}$.

\end{enumerate}
\end{lemma}

\begin{proof}
Let $A$ be a countable \idf\ 
subset of $c$. Notice that if
$x>y$ and $\fracpart x\ne\fracpart y$, then
$$\lfloor |x-y|\rfloor =
\begin{cases}
\lfloor x\rfloor -\lfloor y\rfloor&\text{if $\fracpart x> \fracpart y$;}\\
\lfloor x\rfloor -\lfloor y\rfloor-1&\text{if $\fracpart x<\fracpart y$.}
%\lfloor y\rfloor-\lfloor x\rfloor&\text{if $y>x$ and $\fracpart y>\fracpart x$;}\\
%\lfloor y\rfloor - \lfloor x\rfloor-1&\text{if $y>x$ and $\fracpart y<\fracpart x$.}
\end{cases}
$$
Clearly, analogous statements hold if $y<x$ by swapping the roles of $x$ and $y$.

Suppose that $f\colon A\to B$ is a bijection satisfying the hypotheses (1) and (2). If $x,y\in A$, then
notice that $f(x)_i<f(y)_i$ if and only if $x_i<y_i$, and that $\fracpart{f(x)_i}
<\fracpart{f(y)_i}$ if and only if $\fracpart{x_i}<\fracpart{y_i}$,
so that $\lfloor |x_i-y_i|\rfloor=
\lfloor |f(x)_i-f(y)_i|\rfloor$. 

Note that, if we consider $x\in c$ such that $x_\infty\not\in \Z$,
then it is evident that there is an $M$ such that for
all $i\geq M$, $\lfloor x_{i}\rfloor=\lfloor x_{\infty }\rfloor$. Since $A$ and $B$ satisfy the \idf\ property, the limits of $x$ and $y$ are not integers, and thus
$\lfloor\|x-y\|\rfloor =\max_i \lfloor |x_i-y_i|\rfloor$ and
$\lfloor\|f(x)-f(y)\|\rfloor = \max_i \lfloor |f(x)_i-f(y)_i|\rfloor$, so that
$\lfloor \|f(x)-f(y)\|\rfloor = \lfloor\|x-y\|\rfloor$. \end{proof}

The conditions can be loosely stated as follows. For each $i\in\N$, conditions (1) and (2) state that the  $i$-th coordinate of the sequences have the same integer part, and their remainders are similarly ordered. This implies that the induced coordinate map which maps the $i$-th coordinate of every function in $A$ to the $i$-th coordinate of the image of the function, is a step isometry. Since we are working with the $\ell^\infty$ norm, this immediately implies that $f$ is a step isometry, but the condition is indeed stronger.

The proof of Theorem \ref{mainone} follows directly from the following theorem, which states that all \gec\ geometric 1-graphs obtained from \idf\ sets in $c$ are isomorphic. The unique isomorphism type is the graph $GR(c)$ mentioned in Theorem \ref{mainone}. 

\begin{theorem}\label{main}
Let $V$ and $W$ be dense, countable \idf\ sets in $c$.
If $G$ and $H$ are \gec\ geometric
1-graphs with vertex sets $V$ and $W$, respectively, then $G\cong H$.
\end{theorem}

\begin{proof}
We use the usual technique for proving that two countably
infinite graphs are isomorphic: a back-and-forth argument.
That is, to show that two countably infinite graphs $G$ and $H$ are equivalent, we
inductively construct increasing sequences of finite subsets
$V_t\subseteq V$ and $W_t\subseteq W$, together with a sequence of functions $\varphi_t\colon
V_t\to W_t$ such that $\varphi_t$ is an extension of $\varphi_{t-1}$, and has the
property that for $x,y\in V_t$, $x$ and $y$ are adjacent in $G$ if and only if
$\varphi_t(x)$ and $\varphi_t(y)$ are adjacent in $H$.
Throughout the proof, if $x\in V$, then we denote by
$x'$ its image in $W$ (that is, if $x\in V_t$, then $x'=\phi_t(x)\in W_t$).

We just outline the forth part,
as the argument going back is analogous.
Let $V=\{ x^{(1)},x^{(2)},\dots \}$ and $W=\{ \xi^{(1)}, \xi^{(2)},\dots \}$.
The induction hypothesis requires not only that the partial maps constructed form an isomorphism between the relevant subgraphs, but also that the maps satisfy the coordinate-wise step isometry properties from Lemma \ref{step}. This will ensure that we can indeed invoke the \gec\ condition to extend the map to a new vertex. Namely, the partial maps are step isometries, and thus the images under the partial map of the graph neighbours of a given vertex $v$ are all contained in the 1-ball around the image of $v$. In addition, we maintain that the conditions from Lemma \ref{step} also hold for the limits, $x_\infty$ and $x'_\infty$, so that the partial map also induces a step isometry on the limit elements.

Precisely, we maintain the following inductive hypotheses throughout the construction
for a given $t\geq 1$.

\begin{enumerate}
\item $x^{(t)}\in V_{t}$ and $\xi^{(t)}\in W_{t}$. Also, if $t>0$ then
$V_{t-1}\subseteq V_t$ and $W_{t-1}\subseteq W_t$.

\item The map $\varphi _{t}$ is a graph isomorphism from the subgraph
$G[V_{t}]$ to $H[W_{t}]$.  Moreover, $\varphi_t$ extends $\varphi_{t-1}$.

\item The following holds for all points $x$ and $y$ in $V_t$:

\smallskip\noindent
For all $i\in \Nnn \cup \{ \infty \}$, we have that
\begin{enumerate}
\item $\lfloor x_{i} \rfloor = \lfloor x_{i}^{\prime }\rfloor $

\item $\fracpart{x_{i}}<\fracpart{y_{i}}$ if and only if
$\fracpart{x_{i}^{\prime }}<\fracpart{y_{i}^{\prime }}$.
\end{enumerate}

\end{enumerate}

If the induction hypothesis holds for all $t\in \Nnn$, then by (1) and (2),
the map $\varphi :V\rightarrow W$
defined as $\bigcup\limits_{t\in \mathbb{N}}\varphi _{t}$ is an isomorphism, and the
conclusion follows. Note
that, by Lemma \ref{step}, condition (3) implies that $\varphi_t $ is a
step-isometry from $V_t$ to $W_t$.

Set $V_0=W_0=\emptyset$.

For the induction step, let $x=x^{(t+1)}$. If $x\in V_{t}$, then there
is nothing to be done, and we can proceed with the `back'-step, which involves finding a pre-image for $\xi^{(t+1)}$. If not, then let  $V_{t+1}= V_{t}\cup \{x \}$.
Form the sequences $\ell=(\ell_1,\ell_2\dots )$ and $u=(u_1,u_2,\dots) $ as
follows. For all $i\in \mathbb{N}$,
\begin{eqnarray*}
u_{i} &=&\min \big(\{1\}\cup \{\fracpart{z'_{i}}:z\in V_{t}\text{ and }
\fracpart{z_i}>\fracpart{x_i}\}\big),%
\text{ and} \\
\ell _{i} &=&\max \big( \{0\}\cup\{\fracpart{z'_{i}}:z\in V_{t}\text{ and }
\fracpart{z_i}<\fracpart{x_i}\}\big).
\end{eqnarray*}
Also define
\begin{eqnarray*}
u_{\infty } &=&\min \big(\{1\}\cup
\{\fracpart{z'_{\infty }}:z\in V_{t}\text{ and }\fracpart{z_{\infty
}}>\fracpart{x_{\infty }}\}\big)\text{, and} \\
\ell _{\infty } &=&\max \big(\{0\}\cup\{\fracpart{z'_{\infty }}:z\in V_{t}\text{ and }%
\fracpart{z_{\infty }}<\fracpart{x_{\infty }}\}\big).
\end{eqnarray*}

That is, in each coordinate $i$, the remainder of $x_i$ lies between $\ell_i$ and $u_i$, and no other coordinates from sequences in $V_t$ have remainders (in the $i$-th coordinate) between $\ell_i$ and $u_i$. The elements $\ell_\infty$ and $u_\infty$ are defined so that the remainder of $x_\infty $ lies between $\ell_\infty$ and $u_\infty$. The notation $\ell_\infty$ and $u_\infty$ suggests that these values are the limits of the sequences $(\ell_i)$ and $(u_i)$. We now claim and prove that this is indeed the case.

\begin{claim*} The sequences $u$ and $\ell$ defined above are in
$c$. Moreover,  $\lim_{n\rightarrow \infty } u_{n} =u_{\infty }$ and
$\lim_{n\rightarrow \infty}\ell _{n} =\ell _{\infty }$.
\end{claim*}

%\blue{AQ: I've re-written the proof of the claim to make sure it covers the case
%$u_\infty=1$, the cases where $\fracpart{x_\infty}$ is larger than
%$\fracpart{z_\infty}$ for all $z\in V_t$. As things are currently,
%the situation $\ell_\infty=0$
%is covered already since $V_t$ contains the 0 sequence.}

\begin{proof}[Proof of Claim]
By the \idf\ property, for all distinct $w$ and $z$ in $V_t$, we have
$\fracpart{w_\infty} \ne \fracpart{z_\infty}$. Also $\fracpart{x_\infty}$
is distinct from $\fracpart{z_\infty}$ for each $z\in
V_t$ and $0<\fracpart{x_\infty}<1$. Hence, $F_\infty=
\{\fracpart{z_\infty}\colon z\in V_t\}\cup \{0,1, \fracpart{x_\infty}\}$
consists of $|V_t|+3$ elements.
Since $F_\infty$ contains 0 and 1,  $\fracpart{x_\infty}$ is neither the
largest nor the smallest element of the set. For all sufficiently large $i$,
the ordering of $F_i=\{\fracpart{z_i}\colon z\in V_t\}\cup
\{0,1,\fracpart{x_i}\}$ is identical to the ordering of $F_\infty$
(that is if $z_i$ is the $k$th largest element of $F_i$ for some
$z\in V_t$, then $z_\infty$ is the $k$th largest element of
$F_\infty$). Since $u_i$ is the next largest element of $F_i$ after $x_i$ and
$\ell_i$ is the next smallest element of $F_i$ before $x_i$, we deduce that
$u_i\to u_\infty$ and $\ell_i\to\ell_\infty$
as required.
\end{proof}

To continue the proof of the theorem, we will first form a \emph{target sequence}
$y \in c$. The
idea is to choose a sequence whose remainders are trapped above
$\ell$ and
below  $u $ coordinate-wise. If we would set $\varphi_{t+1}(x)=y$, then the extended function would satisfy condition (3) of the induction hypothesis, i.e.~the map would be a coordinate-wise step isometry. The target sequence may not be
in $W$, but we can use a density argument to find a sequence in $W$ sufficiently
close to our target sequence which satisfies condition (3) of the induction hypothesis, and also condition (2).
 
Let's define the target sequence. For all $i\in \mathbb{N}$, define
\begin{equation*}
y_{i}=\lfloor x_{i}\rfloor +\frac{\ell _{i}+u_{i}}{2}.
\end{equation*}
Then by properties of limits, $y=( y_1,y_2,\dots ) $ converges
to $y_{\infty }=\lfloor x_{\infty }\rfloor +
\frac{\ell _{\infty }+u_{\infty }}{2}$.
As stated above, if we could choose $\varphi_{t+1}(x)=y$, then condition (3)
of the induction hypothesis would be satisfied. However, we need to find an image
that also fulfills the other
conditions. Specifically, the image of $x$ must be in $W$, and it must
be correctly joined to the elements in $W_t$ to maintain the isomorphism condition.
Thus, we will define a ball around
$y$ with the property that each element in this ball will maintain condition (3).
We can then invoke density and the \gec~condition (which holds by
Theorem~\ref{random}) to find another point in this ball that
qualifies for the induction step.

Define $\alpha = \frac{u_{\infty }-\ell _{\infty }}{6}$.
Note that $\alpha <1$ by definition.
Now we have that $|u_{\infty }-\fracpart{y_{\infty }}|=3\alpha $ and
$|\ell _{\infty}-\fracpart{y_{\infty
}}|=3\alpha $. As $\ell$ and $ u $ are
convergent sequences, there is a positive integer $M$ such that for all
$i\geq M$, $|u_{i}-u_{\infty }|<\alpha$,
$|\ell _{i}-\ell _{\infty }|<\alpha$, and
$\lfloor y_i \rfloor =\lfloor x_{i}\rfloor = \lfloor x_{\infty }\rfloor$.
Thus, we also have that
$|\fracpart{y_{\infty }}- \fracpart{y_{i}}|=|y_{\infty }-y_{i}|<\alpha$.
Then for all $i\geq M$ we have by the triangle equality that

\begin{eqnarray*}
|\fracpart{y_i}-u_i| &\geq &|\fracpart{y_{\infty }}-u_{\infty }|-|u_i-u_{\infty
}|-|\fracpart{y_{\infty }}-\fracpart{y_{i}}| \\
&\geq &3\alpha -\alpha -\alpha  \\
&=&\alpha .
\end{eqnarray*}

By a symmetric argument, for all $i\geq M$ we have that
$|\ell_{i}-\fracpart{y_{i}}|\geq \alpha $.
Therefore, for all $i\geq M$ and $t\in \mathbb{R}$ such that $%
|t-y_{i}|<\alpha $, we have that $\ell _{i}<t<u_{i}$.

Next, let $\alpha'<\alpha$ be such that
$$
2\alpha ^{\prime }\leq \min \{u_{i}-\ell _{i}:1\leq i<M\}.
$$
It follows that any point $v\in B_{\alpha'}(y)$,
when taken as the image of $x$, would satisfy property
(3) of the induction hypothesis. 

We now invoke the \gec~condition
to find a sequence in $W$ that is
correctly joined. We do this in two steps. First we identify
the set $A$ of all vertices in
$W_t$ that the image of $x$ should be adjacent to. Then we find a point $v\in
B_{\alpha'}(y)\cap W$ so that $A$ is contained in a unit ball around $v$. Then we invoke the \gec~condition to find a sequence close to $v$ which is adjacent to
all vertices in $A$, and no other vertices in $W_t$.
%So for any such $(v_n)$, and any $(z_n)\in V_t$, $(z_n)\sim (x_n)$ in $G$
%implies that $(z_n)\in B((x_n))$, and thus $(z'_n)\in B(v_n)$.

Precisely, let  $A=\{ z'\in W_t : z\sim x\}$, the set of images of
neighbours of $x$ in $W_t$. Note that $A\subseteq B(y)$. Namely, since $G$
is a geometric 1-graph, each
neighbour $z$ of $x$ lies within $B(x)$. Our construction of $y$
guaranteed that the step-isometry condition is satisfied if $y$ is chosen as the
image of $x$, and thus, the
image $z'$ lies in $B(y)$. Let $\beta >0$ be so that $A\subseteq B_{1-\beta}(y)$,
and $\beta\leq \alpha'/2$. We use the density of $W$ to find a point
$v\in W\setminus W_t$ so
that $\|v-y\|<\beta$. By our choice of $\beta$, $A\subseteq B(v)$.

Next, we invoke the \gec~property to find  a point $w\in W \cap B_{\beta}(v)$
which is adjacent to all vertices in $A$, and to no other vertices in $W_t$.
In particular, for all
$z\in V_{t}$, $z$ is adjacent to $x$ if and only if $z'$ is
adjacent to $w$. Now let $W_{t+1}=W_t\cup \{w\}$, and set
$\varphi_{t+1}(x)=w$. Then condition (2)
holds for $\varphi_{t+1}$. Since $\|w-y\|\leq \|w-v\|+\|v-y\|<
2\beta\leq \alpha'$, condition (3) holds as well.

Similarly, we can find a pre-image for $\xi^{(t+1)}$, and adjust $V_{t+1}$, $W_{t+1}$
accordingly, to satisfy condition (1) and complete the reverse direction of the
back-and-forth argument.
\end{proof}

Using Theorem~\ref{main} and the results of Section~\ref{randomsec},
Theorem~\ref{mainone} follows immediately.

\section{Proof of Theorem~\ref{maintwo}: almost all sets in $c_0$ are Rado}\label{c0sec}

In this section, our goal is to study the Rado property for countable subsets of
$c_0$. It turns out that it is easier to first prove the Rado property for
countable subsets of $c_a$ for $a\in(0,1)$
(because $\lfloor x_i\rfloor$ is eventually 0 for all $x\in c_a$, while no
such statement holds for $c_0$). It is then straightforward to deduce the result for $c_0$.

We prove the following theorem.

\begin{theorem}\label{main2}
Let $a\in (0,1)$. Let $V$ and $W$ be dense, countable, \idf\ sets in $c_{a}$
satisfying the \iop\  If $G$ and $H$ are \gec\ geometric 1-graphs
with vertex sets $V$ and $W$ respectively, then $G\cong H$.
% Then for all
%$p,p'\in (0,1)$, with probability $1$, any two graphs $G$, sampled from
%$\mathrm{LARG}(V,1 ,p)$, and $H$, $\mathrm{LARG}(V',1,p')$, are isomorphic.
\end{theorem}

From Theorem~\ref{main2}, the conclusion of Theorem~\ref{maintwo} follows as we now show.
% AQ Fix here: 1) theorem reference was wrong; 2) the proof follows -> the conclusion follows

\begin{proof}[Proof of Theorem~\ref{maintwo}]
Let $\theta(x)=x+(\frac 12,\frac 12,\ldots)$, so that $\theta$ is a
bijective isometry from $c_0$ to $c_{1/2}$. If $\mu$ is a non-aligned
measure fully supported on $c_0$, then
$\mu'(A)=\mu(\theta^{-1}(A))$ is a non-aligned measure fully supported on
$c_{1/2}$. By Lemmas \ref{lem:idf}, \ref{lem:iop} and \ref{lem:dense},
${\mu'}^\N$-a.e. sequence of points,
${x'}^{(1)},{x'}^{(2)},\ldots$ in $c_{1/2}$ is dense, \idf,\ and satisfies
the \iop, so that $W=\{{x'}^{(1)},{x'}^{(2)},\ldots\}$ is Rado by Theorems
\ref{random} and  \ref{main2}. Since the Rado
property is preserved by isometries, we deduce that $V=\theta^{-1}(W)$ is
Rado for ${\mu'}^\N$-almost every $({x'}^{(n)})$. Hence, for
$\mu^\N$-almost all $x^{(1)},x^{(2)},\ldots$ in ${c_0}^\N$,
$V=\{x^{(1)},x^{(2)},\ldots\}$ is Rado.
\end{proof}

Before proving Theorem \ref{main2}, we comment on the argument, and the differences
with the proof of Theorem \ref{main}. As before, the idea is to use
a back-and-forth argument. Similarly, we use induction hypotheses to guarantee that at
each step of the induction, we have a step isometry between the finite sets of vertices 
that are matched (otherwise Theorem \ref{mot} would show that we will not obtain a
graph isomorphism in the limit). Lemma \ref{step} still provides a sufficient condition
for a matching to be a step isometry, but unfortunately, in the case of $c_a$, it is essentially
impossible to satisfy the conditions of the lemma: if $x$ is a randomly chosen element
of $c_a$, write a sequence of $+$'s and $-$'s to denote if each coordinate is above or 
below $a$. One quickly sees that the probability that two random sets $V$ and $W$ 
contain elements with identical sign sequences is 0 (this is the union of countably many
events of probability 0). This prevents us from directly applying Lemma \ref{step}. 
To circumvent this difficulty, we inductively build not only the bijection between the
sets $V$ and $W$, but also a bijection, $g$ of the coordinates, so the partial bijection
from $V$ to $W$ sends a point $x$ with integer part $k\ne 0$ in position $i$ to a point $y$ with
integer part $k$ in position $g(i)$, and if $\fracpart{x^{(i)}}<\fracpart{{x'}^{(i)}}$, then 
$\fracpart{y^{(g(i))}}<\fracpart{{y'}^{(g(i))}}$. This is sufficient to maintain the 
step isometry property that we require.

\begin{proof}[Proof of Theorem~\ref{main2}]
We show, using a back-and-forth type argument, that $G$ and $H$
are almost surely isomorphic. Enumerate $V$ and $W$
as $V=\{ x^{(1)},x^{(2)},\dots\}$ and
$W=\{ y^{(1)},y^{(2)},\dots\}$ respectively. By assumption, these sets are
\idf\ and have the \iop

We build bijections $f\colon\N\to\N$ and $g\colon\N\to\N$ so that:
\begin{enumerate}
\item[($i$)]
$\lfloor y^{(f(i))}_{g(j)}\rfloor=\lfloor x^{(i)}_j\rfloor$ for each $i$ and $j$;
\item[($ii$)]
$\fracpart{y^{(f(i))}_{g(j)}} < \fracpart{y^{(f(i'))}_{g(j)}}$ if and only if
$\fracpart{x^{(i)}_j} < \fracpart{x^{(i')}_j}$ for each $i$, $i'$ and $j$;
\item[($iii$)] $x^{(i)}$ is adjacent to $x^{(i')}$ if and only
if $y^{(f(i))}$ is adjacent to $y^{(f(i'))}$.
\end{enumerate}

The function $f$ can be transformed into an isomorphism from
$G$ to $H$ by defining $\phi:V\rightarrow W$ by
\begin{equation*}
\phi(x^{(i)})=y^{(f(i))},\text{ for all }i\in \N.
\end{equation*}
Property ($iii$) guarantees that $f$ is indeed an isomorphism; properties ($i$) and ($ii$) guarantee that $f$ is a coordinate-wise step isometry under the map $j\rightarrow g(j)$.

We construct the bijections $f$ and $g$ by induction. That is,
we inductively construct sets $S_n$ and $T_n$, for $n\in \N$,
and partially define $f$ and $g$ on these sets.

The sets $S_n$ will represent the (super-)indices of the points
$(x^{(i)})$ that have so far been matched. The point
$x^{(i)}$ is matched to the point $y^{(f(i))}$. The sets $T_n$
will represent the sets of coordinate positions
(in the $(x^{(i)})$ sequences) that have been matched with
coordinate positions in the
$(y^{(f(i))})$ sequences: the $j$th coordinate of all sequences
$x\in V$ that have been matched are matched with the $g(j)$-th
coordinate of the image $y\in W$ of the sequence.
%There is no re-definition.

The induction hypotheses that we maintain are:
\begin{enumerate}
\item $S_{n}\supseteq \{1,\ldots,n\}$ and $T_{n}\supseteq
\{1,\ldots,n\}$; \label{it:prop1}
\item $f(S_{n})\supseteq\{1,\ldots,n\}$ and $g(T_{n})
\supseteq\{1,\ldots,n\}$; \label{it:prop2}
\item $T_{n}\supseteq\{j\colon \mbox{there exists } i\in
S_{n}\colon x^{(i)}_j\not\in(0,1)\}$;
\label{it:largexjs}
\item $g(T_{n})\supseteq\{j\colon \mbox{there exists } i\in
f(S_{n})\colon y^{(i)}_j \not\in(0,1)\}$; \label{it:largeyjs}
\item $\fracpart{x^{(i)}_k}<\fracpart{x^{(j)}_k}$ if and only if
$\fracpart{y^{(f(i))}_{g(k)}}<\fracpart{y^{(f(j))}_{g(k)}}$ for
$i,j\in S_{n}$ and $k\in T_{n}$. \label{it:prop3}
\item
$\lfloor y^{(f(i))}_{g(k)}\rfloor=\lfloor x^{(i)}_k\rfloor$
for each $i\in S_n$ and $k\in \N$;\label{it:prop4}
\item $x^{(i)}\sim x^{(j)}$ if and only if $y^{(f(i))}\sim y^{(f(j))}$
for $i,j\in S_{n}$. \label{it:goodconn}
\end{enumerate}
Conditions (\ref{it:prop1}) and (\ref{it:prop2}) ensure that
$\bigcup_{n\in N} S_n=\N$ and $\bigcup_{n\in N} T_n=\N$, and thus,
that $f$ and $g$ are bijections of $\N$, as required. Conditions
(\ref{it:prop3}) and (\ref{it:prop4}) ensure that $f$ and $g$ satisfy
Properties ($i$) and ($ii$) above, and Condition (\ref{it:goodconn})
ensures that $f$ and $g$ satisfy Property ($iii$).
Conditions (\ref{it:largexjs}) and (\ref{it:largeyjs}) are
necessary to ensure that the induction can be propagated.

We now proceed with the induction. Let $S_0=T_0=\emptyset$.
Suppose that for a given $n\ge 1$, we have constructed $S_{n-1}$ and $T_{n-1}$,
and $f$ and $g$ are defined on those sets, and that the inductive properties hold.
We now show how to extend the sets $S_{n-1}$ and $T_{n-1}$ in such a way
that these properties are preserved.  We focus on making an extension
satisfying properties \eqref{it:prop1},
\eqref{it:largexjs}, \eqref{it:prop3}, \eqref{it:prop4}
and \eqref{it:goodconn} (that is, we go forth).
A symmetric (and so omitted) ``going back'' argument applies for the extension
properties \eqref{it:prop2},
\eqref{it:largeyjs}, \eqref{it:prop3}, \eqref{it:prop4} and \eqref{it:goodconn}.

Here we match $x^{(n)}$ (if it is not already matched) with a
sequence $y^{(i)}$, and we match some
coordinate positions (including the $n$-th if it is not already matched,
to ensure that the
final map $g$ is a bijection) with some
coordinate positions in the $y$'s.

First, set $S_n=S_{n-1}\cup\{n\}$ and
$T_{n}=T_{n-1}\cup\{n\}\cup\{j\colon x^{(n)}_j\not\in(0,1)\}$;
this satisfies property (\ref{it:largexjs}). Let $T_{n}\setminus T_{n-1}=
\{j_1,\ldots,j_L\}$. Notice that $T_{n}$ is a finite
set since $x^{(n)}\in c_{a}$ and $0<a<1$. For each $k$, $1\leq k\leq L$,
define an order $\prec_k$ on $S_{n-1}$ by $i\prec_k i'$ if
$\fracpart{x^{(i)}_{j_k}}<\fracpart{x^{(i')}_{j_k}}$. Let
% AQ Fix here: it's r_1,...,r_L; not r_1...r_k
$r_1,\ldots,r_L$ be coordinate positions not in $g(T_{n-1})$ such that,
for all $i,i'\in S_{n-1}$ and each $1\le k\le L$, $\fracpart{y^{(f(i))}_{r_k}}<
% AQ Fix here: added each 1<=k<=L
\fracpart{y^{(f(i'))}_{r_k}}$ if and only if $i\prec_k i'$. Such coordinate
positions exist by the \iop\

Now define $g(j_k)=r_k$ for
$k=1,\ldots,L$.
Notice that condition \eqref{it:prop4} is automatically
satisfied for $i\in S_{n-1}$
and $k\in T_n\setminus T_{n-1}$ by conditions
\eqref{it:largexjs} and \eqref{it:largeyjs}
of the induction hypothesis
since one has $\lfloor x^{(i)}_k\rfloor = 0
=\lfloor y^{(f(i))}_{g(k)}\rfloor$.
This has extended the collection of dimensions.
Notice that property \eqref{it:prop3} is automatically
satisfied (for $j\in T_{n}$) by the choice of $r_1,\ldots,r_L$
provided that $i$ and $i'$ belong to $S_{n-1}$.

To finish the argument, we need to choose $f(n)$ (assuming that
$n\not\in S_{n-1}$). For $j\in T_n$, let $k=g(j)$, let
$a_k=\lfloor x^{(n)}_j\rfloor$, let
$b_k=\max\big(\{0\}\cup \{ \fracpart{y^{(f(i))}_k}\colon i\in S_{n-1},
\fracpart{x^{(i)}_j}<\fracpart{x^{(n)}_j}\}\big)$ and
$c_k=\min\big(\{1\}\cup \{\fracpart{y^{(f(i))}_k}\colon i\in
S_{n-1},\fracpart{x^{(i)}_j}> \fracpart{x^{(n)}_j}\}\big)$.
Let $J_k$ be the open interval $(a_k+b_k,a_k+c_k)$. For $k\not\in g(T_n)$,
let $J_k=(0,1)$ and let $U=c_a\cap \prod_{k\in\N}J_k$. Notice that $U$ contains an
open set of radius $\frac 12\min_{k\in g(T_n)}(c_k-b_k)$.

\begin{claim*}
If $z\in U$, then $\lfloor\|z-y^{(f(i))}\|\rfloor=\lfloor\|x^{(n)}-x^{(i)}\|\rfloor$
for each $i\in S_{n-1}$.
\end{claim*}

\begin{proof}
Let $z\in U$.
Since all points belong to $c_a$, we have $\|z-y^{(f(i))}\|=
\max_k |z_k-y^{(f(i))}_k|$
and $\|x^{(n)}-x^{(i)}\|=\max_j |x_j^{(n)}-x^{(i)}_j|$.
In particular, it suffices to show that
$\max_k\lfloor |z_k-y^{(f(i))}_k|\rfloor=\max_j\lfloor|x^{(n)}_j-x^{(i)}_j|\rfloor$
for each $i\in S_{n-1}$.

Notice that if $k\not\in g(T_n)$, then $z_k,y^{(f(i))}_k\in (0,1)$ for each
$i\in S_{n-1}$ by Condition
\eqref{it:largeyjs}, so that $\lfloor |z_k-y^{(f(i))}_k|\rfloor=0$. Similarly if
$j\not\in T_n$, then $\lfloor |x^{(n)}_j-x^{(i)}_j|\rfloor=0$ for each $i\in S_{n-1}$
by Condition \eqref{it:largexjs}.

% *** AQ Fix here: replace x by x^{(n)}
Hence, it suffices to show that
$\lfloor |z_{g(j)}-y^{(f(i))}_{g(j)}|\rfloor=\lfloor |x^{(n)}_j-x^{(i)}_j|\rfloor$
for each $i\in S_{n-1}$ and each $j\in T_n$. First notice that $\lfloor
z_{g(j)}\rfloor =\lfloor x^{(n)}_j\rfloor$ for each $z\in U$ and
$\lfloor y^{(f(i))}_{g(j)}\rfloor = \lfloor x^{(i)}_j\rfloor$ for each
$i\in S_{n-1}$ and $j\in T_n$ by Conditions \eqref{it:largexjs},
\eqref{it:largeyjs} and \eqref{it:prop4}. Also
$\fracpart{z_{g(j)}}<\fracpart{y^{(f(i))}_{g(j)}}$ if and only if
$\fracpart{x^{(n)}_j}<\fracpart{x^{(i)}_j}$ by the choice of
$(b_k)_{k\in g(T_n)}$ and
$(c_k)_{k\in g(T_n)}$. An argument exactly analogous to the proof of
Lemma \ref{step} establishes that
$\lfloor |z_{g(j)}-y^{(f(i))}_{g(j)}|\rfloor=\lfloor |x^{(n)}_j-x^{(i)}_j|\rfloor$ for each
$i\in S_{n-1}$ and $j\in T_n$, completing the proof.
\end{proof}

Now let $A=\{i\in S_{n-1}\colon \|x^{(n)}-x^{(i)}\|<1\}$. By the above claim,
if $i\in A$, then $\|z-y^{(f(i))}\|<1$ for each $z\in U$.
Since $H$ satisfies the \gec\ property and $U$ contains an open set, there exists
$y^{(t)}\in U\cap H$ with $t\not\in f(S_{n-1})$ such that
$y^{(t)}\sim_H y^{(f(i))}$ if and only if $x^{(n)}\sim_G x^{(i)}$.
Defining $f(n)=t$ ensures that the induction hypotheses are
satisfied at the $n$th stage as required.
\end{proof}

\section{non-isomorphism of Rado graphs}
\label{sec:noniso}

Let $GR(c)$ and $GR(c_0)$ be the unique isomorphism type of the geometric Rado graphs provided by Theorems \ref{main} and \ref{main2}, respectively. Similarly, let $GR(\ell^d_\infty)$ be the unique
isomorphism type of the geometric Rado graph on $\ell^d_\infty$ provided by \cite{bona}. We show in this section that these graphs are mutually non-isomorphic. In fact, we show more generally that a
geometric graph on a countable dense subset of a Banach space determines the space up to isometric isomorphism.

To show this result, we need to make use of a concept introduced in the 1940's by Hyers and Ulam \cite{HyersUlam} and developed over the next five decades by a number of authors. An
$\epsilon$\emph{-isometry} from a Banach space $X$ to a Banach space $Y$ is a map $T\colon X\to Y$ such that $\Big|\|T(x)-T(x')\|-\|x-x'\|\Big|\le \epsilon$ for all $x,x'\in X$. A map $T$ from a
Banach space $X$ to a Banach space $Y$ is $\delta$\emph{-surjective} if for all $y\in Y$, there exists $x\in X$ with $\|T(x)-y\|<\delta$.

The following theorem states that $\delta$-surjective $\epsilon$-isometries of
Banach spaces are uniformly approximated by genuine isometries.

\begin{theorem}[Dilworth \cite{dilworth}, Proposition 2]\label{dil}
Let $X$ and $Y$ be real Banach spaces and let $T$ be a $\delta$-surjective
$\epsilon$-isometry from $X$ to $Y$. Then there
exists a $K<\infty$ and a surjective linear isometry $U\colon X\to Y$ such that
$\|T(x)-U(x)\|\le K$ for all $x\in X$.
\end{theorem}

\begin{theorem}
Let $X$ and $Y$ be real Banach spaces and suppose $V$ and $W$ are countable dense subsets of $X$ and $Y$ respectively. Further, let $G$ and $H$ be geometric 1-graphs with vertex sets $V$ and $W$. If
$G$ and $H$ are isomorphic, then there is an surjective isometry between $X$ and $Y$.
\end{theorem}

\begin{proof}
Let $\theta$ be the isomorphism from $G$ to $H$. Let $V$ be enumerated as
$\{x^{(1)},x^{(2)},\ldots\}$ and let $y^{(n)}=\theta(x^{(n)})$. By Theorem~\ref{mot},
we have that for $i\ne j$,
$d_G(x^{(i)},x^{(j)})<\|x^{(i)}-x^{(j)}\|+2$, so that $\|y^{(i)}-y^{(j)}\|
<d_H(y^{(i)},y^{(j)})=d_G(x^{(i)},x^{(j)})< \|x^{(i)}-x^{(j)}\|+2$; and similarly
$\|x^{(i)}-x^{(j)}\|<\|y^{(i)}-y^{(j)}\|+2$, so that
$\Big|\|\theta(x^{(i)})-\theta(x^{(j)})\|_Y-\|x^{(j)}-x^{(i)}\|_X\Big|<2$.
We show how to extend $\theta$ to a 3-surjective 4-isometry from $X$ to
$Y$. Namely, we define $n\colon X\to\mathbb N$ by $n(x)=\min\{i\colon \|x^{(i)}-x\|<1\}$.
We now define $T(x)=y^{(n(x))}$. We claim that $T$ is a 3-surjective 4-isometry.

To see this, first let $x,x'\in X$. Let $i=n(x)$ and $i'=n(x')$. We then have that
$\|T(x)-T(x')\|=\|\theta(x^{(i)})-\theta(x^{(i')})\|$, so that
% **** AQ FIX HERE: replace T with theta
$$
\Big|\|T(x)-T(x')\|-\|x^{(i)}-x^{(i')}\|\Big|<2.
$$
We also have that
$$
\Big|\|x^{(i)}-x^{(i')}\|-\|x-x'\|\Big|<2.
$$
We deduce that
$$
\Big|\|T(x)-T(x')\|-\|x-x'\|\Big|<4,
$$
so that $T$ is a 4-isometry.

Next, let $y\in Y$. Then there exists a $j$ such that $\|y-y^{(j)}\|<1$. We consider $T(x^{(j)})$: let $i=n(x^{(j)})$. Notice that $\|x^{(i)}-x^{(j)}\|<1$ by definition of $n(\cdot)$. By Theorem
\ref{mot}, we have that $\|y^{(i)}-y^{(j)}\|<2$. Hence, $T(x^{(j)})=y^{(i)}$ and $\|y^{(i)}-y\|<3$, so that $T$ is 3-surjective as required. Theorem~\ref{dil} then implies that $X$ and $Y$
are linearly isometric as Banach spaces.
\end{proof}

The spaces $\ell_\infty^d$ for $d\in\N$ (that is, $\R^d$ equipped with the $\ell_\infty$ norm) are evidently not isometrically equivalent to each other or to $c_0$ or $c$ since surjective isometries
preserve dimension. Further, $c$ and $c_0$ are not isometrically equivalent. One way to see this is as follows; see \cite{bryant} for additional background. The closed unit ball of $c$ has
\emph{extreme points} (that is, points that are not a convex combination of other points in the set; for example, $(1,1,\ldots)$). There are no extreme points of the closed unit ball of $c_0$: for
any $x$ in the closed unit ball of $c_0$, there is an $n$ such that $|x_n|<\frac 12$. Then $x$ is equal to $\frac 12(y+z)$, where $y$ and $z$ agree with $x$ in every position except the $n$th;
$y_n=x_n+\frac 12$ and $z_n=x_n-\frac 12$. However, it is evident that isometric equivalences preserve unit balls and extreme points.

\begin{corollary}\label{cornon}
The graphs $GR(\ell_\infty^d)$ for $d\in\N$, $GR(c)$ and $GR(c_0)$ are mutually non-isomorphic.
\end{corollary}

We finish with some open problems. We would like to consider other infinite dimensional separable spaces, such as the space of continuous functions $C([0,1])$ on $[0,1]$. Are almost all countable
sets in $C([0,1])$ Rado with respect to some natural measure? The abundance (or non-abundance) of Rado sets for other normed spaces such as $\ell_p$ and $L^p$ (where $1<p<\infty$) remains open.

\section{Appendix}

In this appendix, we show that dense countable \idf\ sets are abundant, in the sense
that we give a probabilistic construction that almost surely yields a dense countable
\idf\ subset of $c$. We also probabilistically construct dense countable \idf\ subsets
of $c_a$ with the \iop\ 
We remark that we make extensive use of the Borel-Cantelli lemmas. We recall that the
first Borel-Cantelli lemma states that if $(A_n)$ is a countable collection of events
in a probability space satisfying $\sum_n\PP(A_n)<\infty$, then almost surely the 
number of those events that occur is finite. The second Borel-Cantelli lemma is a 
partial converse, stating that if $(A_n)$ is a countable collection of independent
events such that $\sum_n\PP(A_n)=\infty$, then almost surely the number of events
that occur is infinite. 

An \emph{affine coordinate hyperplane} is a
subset of $c$ (or
$c_a$) of the form $\{x\in c\colon x_j=\alpha\}$ for some $j\in\N$ and some
$\alpha\in\R$. A \emph{limit coordinate hyperplane} is a subset of $c$ of the form
$\{x\in c\colon \lim_{j\to\infty}
x_j=\alpha\}$ for some $\alpha\in\R$.
%Such an affine coordinate hyperplane is
%\emph{non-trivial} if $\alpha\ne 0$.

A measure $\mu$ on $c$ is said to be \emph{non-aligned} if $\mu(H)=0$ for any affine
coordinate hyperplane or limit coordinate hyperplane. A measure $\mu$ on $c_a$ is
said to be \emph{non-aligned} if
$\mu(H)=0$ for any affine coordinate hyperplane.
%A measure $\mu$ on $\whco$ is said to be \emph{non-aligned} if $\mu(H)=0$
% for any non-trivial affine coordinate hyperplane.
A measure $\mu$ on the space $X$ ($c$ or $c_a$) has
the \emph{integer-distance free} (or \idf) property if, for all $n\in
\N$, when $X=c_a$, or for all $n\in\N\cup\{\infty\}$, when $X=c$, it
holds that $\mu\{ x \in X\colon x_n\in \Z\}=0$ and
$\mu\times\mu\{(x,y)\in X\times X\colon x_n-y_n\in \Z\}=0$.
Intuitively, a
countable subset of $X$ chosen according to an \idf\ measure $\mu$
should give $\mu$-almost surely an i.d.f.\ set. We will show below
that this is indeed the case in Lemma~\ref{lemm}.

Given a measure $\mu$ on $X$ (one of $c$ and $c_a$), we can build a measure
$\mu^\N$, so that sampling a sequence from this measure $(Z_1,Z_2,\ldots)$
one obtains a sequence of elements of $X$, where each coordinate is 
independently chosen with distribution $\mu$.

A measure $\mu$ on a topological space $X$ is said to be \emph{fully supported}
if $\mu(U)>0$ for each non-empty open set $U$ (or equivalently if $\mu(B)>0$ 
for each non-empty open ball).

\begin{lemma} \label{lemm} If $\mu$ is a measure on the space $X$ (one of $c$ or $c_a$)
with the \idf\ property,
then $\mu^\N$-a.e.\ element of
$X^\N$ is a set with the \idf\ property.
%has the property that $\|x^{(i)}-x^{(j)}\|\not\in\N$ for each $i$ and $j$ distinct.
\end{lemma}

\begin{proof}
The set of elements $(x^{(1)},x^{(2)},\ldots)$ of $\Omega^\N$ where there is some pair $x^{(i)}$ and $x^{(j)}$ such that, for some $n\in \N\cup \{ \infty \} $, (or $n\in \N$ in case where
$X=c_a$), satisfying $x^{(i)}_n\in\Z$ or $x^{(i)}_n-x^{(j)}_n\in\Z$, is the countable union of sets of measure 0.
\end{proof}

In turn, we argue that it is sufficient to have a non-aligned measure
in each of $c$ and $c_0$.

\begin{lemma}\label{lem:idf}
If $\mu$ is a non-aligned measure on $c$ or $c_a$, then $\mu$ has the \idf\ property.
\end{lemma}

\begin{proof}
We deal with the two spaces in turn, although the proofs are nearly identical.

First if $X=c$, we observe that $\mu$ has the \idf\ property if
$\mu\times\mu(L)=0$, $\mu(B_j)=0$ for each $j$
and $\mu\times\mu(C_j)=0$ for each $j$, 
where
$L=\{(x,y)\in c\times c\colon\lim_j (x_j-y_j)\in\Z\}$,
$B_j= \{x\in c\colon x_j\in\Z\}$, and $C_j=
\{(x,y)\in c\times c\colon x_j-y_j\in\Z\}$.

\begin{align*}
\mu\times\mu(L)&=\sum_{n\in\Z}\int\mathbf 1_{\{\lim_{j}(x_j-y_j)=n\}}
\,d(\mu\times\mu)
(x,y)\\
&=\sum_{n\in\Z}\int\left(\int\mathbf 1_{\{\lim_{j}(x_j-y_j)=n\}}
\,d\mu(x)\right)\,d\mu(y)\\
&=\sum_{n\in\Z}\int\left(\int\mathbf 1_{\{\lim_{j}x_j=\lim_j y_j+n\}}
\,d\mu(x)\right)\,d\mu(y),
\end{align*}
where we used Fubini's theorem. By the non-aligned property, the inner
integral is 0 for
each value of $\lim_j y_j+n$, and so $\mu\times\mu(L)=0$. An exactly
similar argument shows that
$\mu(B_J)=0$ and $\mu\times\mu(C_j)=0$ for each $j$.

In the case $X=c_a$, there is no $L$ to consider. Define
$C_j$ as above, but applied to elements of $c_a$.
Redefine $B_j= \{x\in c_a\colon x_j\in\Z\}$.
Now  $\mu$ has the \idf\ property if $\bigcup_j B_j\cup\bigcup_j C_j$
has measure zero.
The countable union of the $C_j$'s and $B_j$'s  has measure 0 by a similar argument.
\end{proof}

For $c_0$, our vertex sets will need to satisfy the \iop\ This can be achieved by imposing an additional condition on the measure, which we define here. A measure $\mu$  on subsets of $\R^N$ is of
{\em product type} if it is of the form $\mu=\prod_{n=1}^\infty \nu_n$ where $(\nu_n)_{n\in\N}$ is a collection of measures on $\R$ (that is, the law of $\mu$ is a sequence of independent random
variables, where the $j$th coordinate has law $\nu_j$).

A measure $\nu$ on $\R$ is said to be \emph{non-atomic} if $\nu(\{a\})=0$ for each $a\in \R$. Note that a measure of product type is non-aligned if and only if each $\nu_n$ is non-atomic.

\begin{lemma}\label{lem:iop}
 If $\mu $ is a non-aligned measure of product type, then $\mu^\N$-almost every sequence of points in $\R^\N$ has the \iop
\end{lemma}

\begin{proof}
Fix positive integers $k$ and $i_1<i_2<\ldots<i_k$, and let
$\prec$ be an ordering on $\{1,\ldots,k\}$. Let $m_1,\ldots,m_k$
be the enumeration of $\{1,\ldots,k\}$ in increasing $\prec$ order. We
consider realizations of $\mu^\N$, which we denote as
$(x^{(i)})_{i\in\N}$, where each $x^{(i)}$ is an element of $\R^\N$.

Let $E_j=\{(x^{(i)})\in(\R^\N)^\N\colon \fracpart{x^{(i_{m_1})}_j} <
\ldots < \fracpart{x^{(i_{m_k})}_j}\}$, that is, the collection of those
realizations such that in the $j$th coordinate, the order among the fractional
parts of $x^{(i_1)},\ldots,x^{(i_k)}$ is $\prec$.
Since $x^{(i_1)}_j,\ldots,x^{(i_k)}_j$ are chosen independently,
we see that $\mu^\N(E_j)=1/k!$. Further, since $\mu$ is a product
measure, we see that the events $(E_j)_{j=1}^\infty$ are mutually independent.

By the second Borel-Cantelli lemma, $\mu^\N$-almost every element of $(\R^\N)^\N$
belongs to infinitely many $E_j$'s. Since this holds for each of the $n!$ choices of
$\prec$, it follows that for
$\mu^\N$-a.e.\ choice of sequence $x^{(1)},x^{(2)},\ldots$ of points in $\R^\N$, one
can find a finite sequence of coordinates on which the ordering of the fractional
parts matches any chosen finite
sequence of orderings. Hence, $\mu^\N$-almost every element of $(\R^\N)^\N$
satisfies the \iop, as required.
\end{proof}

\begin{lemma}\label{lem:dense}
Let $X$ be a separable Banach space and let $\mu$ be a fully supported measure on
$X$. Then for $\mu^\N$-a.e. sequence $(x^{(i)})_{i\in\N}$ of points in $X$, the
set $V=\{x^{(i)}\colon i\in\N\}$ is a
dense subset of $X$.
\end{lemma}

\begin{proof}
Let $(z_n)$ be a fixed countable dense subset of $X$, so that $\{B_{1/m}(z_n)\colon
m,n\in\N\}$ is a countable neighbourhood basis of $X$. For fixed $m$ and $n$,
by the second Borel-Cantelli
lemma, with probability 1, the random set, $V$, intersects $B_{1/m}(z_n)$
infinitely often since $\mu(B_{1/m}(z_n))>0$. Let $E_{m,n}$ be the event that 
$V$ intersects $B_{1/m}(z_n)$. Intersecting this countable
collection of events of measure 1 over all choices for $n$, one deduces that with probability 1, the
random set, $V$ intersects every $B_{1/m}(z_n)$. That is, $V$ is a dense subset of
$X$.
\end{proof}

We now construct some non-aligned fully supported measures on $c$ and $c_a$.
We base the construction on normal random variables, but we point out that
there is nothing delicate
about the construction: all that is required is that the distributions of all each coordinate
is non-atomic and fully supported on $\R$, 
but that the distributions of the successive coordinates become more
and more concentrated near $a$ to ensure that with probability 1, the sequence
of coordinates approaches $a$.

First, we construct a measure on $c_a$. Let $(Y_n)_{n\in\N}$ be a family of independent normal random variables with mean $a$
and variance $(\log n)^{-2}$ and let $\mu_a$ be the law of $(Y_1,Y_2,\ldots)$. (That is, $\mu_a(A)= \mathbb P\Big((Y_1,Y_2,\ldots)\in A\Big)$, for each Borel subset, $A$, of $X$). Note that $\mu_a$
is of product type. Since the random variables $(Y_j)_{j\in\N}$ are independent, $\mu_a$ is
exactly as required by the hypothesis of Lemma \ref{lem:iop}, so that $\mu_a^\N$-almost
every sequence in $(c_a)^\N$ has the \iop

\begin{lemma}
The measure $\mu_a$ as constructed above is non-aligned and
fully supported on $c_a$.
\end{lemma}

\begin{proof}
To prove that $\mu_a$ is fully supported on $c_a$, we must prove two things:
firstly, that $\mu_a(c_a)=1$, and secondly, that $\mu_a(U)>0$ for any non-empty
open ball in $c_a$. We work in the case
$a=0$. For any $a\ne 0$, the law of $(Y_1-a,Y_2-a,\ldots)$ is just $\mu_0$,
so that if $\mu_0$ is fully supported on $c_0$, then $\mu_a$ is fully supported
on $c_0+(a,a,\ldots)=c_a$.

To prove that $\mu_0(c_0)=1$, we show for $\mu_0$-almost every sequence
$(x_1,x_2,\ldots)$ that $x_n\to 0$. To see this, write $N_j=(\log j)Y_j$.
Having done this, the random variables $(N_j)$ are
independent standard normal random variables.

Fix $\epsilon>0$. Then
$\mu_0(\{x\colon\limsup_j|x_j|>\epsilon\})\le \PP(\{|N_j|>\epsilon \log j
\text{ infinitely often}\})$.

We have by a standard estimate on the
tail of the normal distribution that
$\PP(N_j>\epsilon\log j)<(2\pi)^{-1/2}\exp(- \epsilon^2(\log j)^2/2)/(\epsilon\log j)$ (see \cite{durrett}, Theorem 1.2.3).
Thus this probability goes to zero as $j$ increases, which tells us that
$\mu_0(\{x\colon\limsup_j|x_j|>\epsilon\})=0$
for each $\epsilon>0$. Now taking a countable sequence of $\epsilon$'s converging to 0
and taking the union, we deduce that
%AQ Fix: ++(taking the union)
$\mu_0(\{x\colon\limsup_j|x_j|=0\})=1$, so that $\mu_0(c_0)=1$.

Now let $U=B(x,r)$ be an open ball in $c_0$. Since $x\in c_0$, we have
$|x_j|<\frac r2$ for all $j$ greater than some $N$. Now
$U\supseteq \Big(\prod_{j=1}^N(x_j-r,x_j+r)\times \prod_{j>N}(-\frac
r2,\frac r2)\Big)\cap c_0$. The above estimate on the tail of a normal distribution shows
that $\prod_{j>N}\PP(|Y_j|<\frac r2)>0$. Since normal random variables are fully
supported on $\R$, we also have $\prod_{j\le N}
\PP(Y_j\in (x_j-r,x_j+r))$ is the product of a finite number of positive reals,
and hence is positive. Hence, we deduce $\mu_0(U)>0$ as required.

Finally, if $H$ is a coordinate hyperplane, $\{x\colon x_j=\alpha\}$, then
$\mu_0(H)=\PP(Y_j=\alpha)=0$, so that $\mu$ is non-aligned.
\end{proof}

To build a fully supported measure on $c$, we start from the previous construction. 
Let $(Y_n)$ be as above and let $Z$ be a standard normal random variable (or any 
fully supported non-atomic random variable). 
Then by the above, the sequence $(Y_1+Z,Y_2+Z,\ldots)$ almost surely takes values in $c$
and the limiting value of such a random sequence is just the random variable $Z$.
Let $\mu$ be the distribution of this random element of $c$.

\begin{lemma}
The measure, $\mu$, as constructed above is non-aligned and fully supported on $c$.
\end{lemma}

\begin{proof}
As mentioned above, sequences sampled from $\mu$ almost surely lie in $c$,
so that the support of $\mu$ is a subset of $c$.
If $U=B(x,r)$ is an open ball in $c$, let $x_\infty= \lim_j x_j$
and let $(y_j)$ be defined by
$y_j=x_j-x_\infty$, so that $y\in c_0$. 
Now if $(Y_1,Y_2,\ldots)$ lies in $B_{c_0}(y,\frac r2)$ and
$Z$ lies in $(x_\infty-\frac r2,x_\infty+\frac r2)$, then 
$(Y_n+Z)_n$ lies in $U$. Since the events that 
$(Y_n)_n$ lies in $B_{c_0}(y,\frac r2)$ and $Z$ lies in $(x_\infty-\frac r2,
x_\infty+\frac r2)$ are independent and both of positive
measure, we see that $\mu(U)$ is bounded below by the product of the measures,
and hence is positive,
proving that $\mu$ is fully supported.

We now verify that $\mu$ is non-aligned.
If $H=\{x\in c\colon x_j=\alpha\}$, then $\mu(H)=\PP(Y_j+Z=\alpha)$. Since $Y_j$ and $Z$ are independent normal random variables, their sum is another normal random variable. Since normal random
variables are non-atomic, we see $\mu(H)=0$.

Similarly, if $L=\{x\in c\colon\lim_j x_j=\alpha\}$, then the event $L$
is equal to the event  $\{Z=\alpha\}$ up to a set of measure 0.
Since $Z$ is non-atomic, we see $\mu(L)=0$, so that $L$ is non-aligned.
\end{proof}

\section*{Acknowledgements}
We thank Keith Taylor and Javad Mashregi for helpful discussions.

\end{document}